\documentclass{amsart}

 \usepackage{eqnarray}
\usepackage[english]{babel}
\usepackage{graphicx}
\usepackage{amsmath,amssymb,amsfonts}
\numberwithin{equation}{section}

\usepackage[colorlinks]{hyperref}

\newcommand{\lj}{[\![}
\newcommand{\rj}{]\!]}
\newcommand{\blj}{\Big[\!\!\Big[}
\newcommand{\brj}{\Big]\!\!\Big]}
\newcommand{\intv}{\tilde{v}}
\newcommand{\ext}{\mathcal{E}}
 \newcommand{\exth}{\mathcal{E}_{h}}

\newcommand\osc{{\rm osc}}

\newtheorem{assumption}{Assumption}[section]
\newtheorem{lemma}{Lemma}[section]
\newtheorem{theorem}{Theorem}[section]

\title[ERROR ANALYSIS OF NITSCHE'S METHOD FOR ROBIN BOUNDARY CONDITIONS]{AN IMPROVED A PRIORI ERROR ANALYSIS OF NITSCHE'S METHOD FOR ROBIN BOUNDARY CONDITIONS}

\author{Nora L\"uthen}\address{Institute of
Mathematics and Systems Analysis, Aalto University--School of Science, P.O. Box 11100,
00076 Aalto, Finland}\email{luethen.n@online.de}

\author{Mika Juntunen}\address{Institute of
Mathematics and Systems Analysis, Aalto University--School of Science, P.O. Box 11100,
00076 Aalto, Finland}\email{mika.juntunen@kone.com}

\author{Rolf Stenberg}\address{Institute of
Mathematics and Systems Analysis, Aalto University--School of Science, P.O. Box 11100,
00076 Aalto, Finland}\email{rolf.stenberg@aalto.fi}

\begin{document}

\maketitle


\begin{abstract} In a previous paper \cite{just} we have extended Nitsche's method \cite{nitsche} for the Poisson equation with general Robin boundary conditions. The analysis required that the solution is in $H^{s}$, with $s>3/2$. Here we give an improved error analysis using a technique proposed by Gudi \cite{gudi}.
 \end{abstract}

 
\section{The method and its consistency}
\label{s:deriving}

In the article \cite{just} a Nitsche-type method is introduced and analyzed for the following model Poisson problem with general Robin boundary conditions:
 Find $u\in H^{1}(\Omega) $ such that 
\begin{align}
-\Delta u & = f \quad \text{in } \Omega, \label{e:modelproblem}
\\
\frac{\partial u}{\partial n} & = \frac{1}{\epsilon}(u_0-u) + g
\quad \text{on } \Gamma, \label{e:modelbcond}
\end{align}
where $\Omega\subset \mathbb{R}^N,\, N=2,3,$ is a bounded domain with polygonal or polyhedral boundary $\Gamma$, $f \in
L^2(\Omega)$, $u_0 \in H^{1/2}(\Gamma)$, $g \in L^2(\Gamma)$, and
$\epsilon \in \mathbb{R}$, $0\leq \epsilon\leq \infty$. The limiting
values of the  parameter $\epsilon$ give  the Dirichlet and
Neumann problems, respectively.

The error analysis presented was not entirely satisfactory. It assumed that the solution is in $H^{s}(\Omega)$ with $s>3/2$, 
which is the same condition that traditionally has been needed for discontinuous Galerkin methods \cite{ern}. 
For discontinuous Galerkin methods Gudi introduced a technique using a posteriori error analysis by which this assumption could be avoided \cite {gudi}. 

The purpose of this paper is to use these arguments to improve the analysis of the Nitsche method for the above Robin problem. 
Below we start by recalling the method of \cite{just}. We first recall the derivation of the method in a way that emphasizes the use of the residual, 
which will be crucial for the error analysis. The same notation as in \cite{just} will be used.
The finite element partitioning into simplexes is denoted by 
 $\mathcal{T}_h$. This induces a mesh, denoted by
$\mathcal{G}_h$,  on the boundary $\Gamma$. By $K \in \mathcal{T}_h$
we denote an element of the mesh and by $E$ we denote an edge or
 a face in $\mathcal{G}_h$.  By $h_K$ we denote the diameter of the
element $K \in \mathcal{T}_h$, and by $\rho_{K}$ the radius of the biggest ball contained in $K$. The mesh is assumed to be regular, i.e. it holds
\begin{equation}
\sup_{K\in \mathcal{T}_h}\frac{h_{K}}{\rho_{K}}=\kappa < \infty.
\end{equation}
By $h_E$ we denote the diameter of
$E \in \mathcal{G}_h$. The finite element subspace is denoted by
\[
V_h :=  \ \{ v \in H^1(\Omega) \, : \, v\vert _{K} \in \mathcal{P}_p(K)
\ \forall K \in \mathcal{T}_h\, \},
\]
where $\mathcal{P}_p(K)$ is the space of polynomials of degree $p$.
 
 The Nitsche method is obtained as follows.
Multiplying the differential equation \eqref{e:modelproblem} with a testfunction $w\in V_{h}$ and integrating by parts we have
\begin{equation}\label{I}
 \big(\nabla u, \nabla w\big)_{\Omega}-\big<\frac{\partial u}{\partial n}, w\big>_{\Gamma}- \big(f, w\big)_{\Omega}=0.
\end{equation}

 Defining the  residual  
\begin{equation}
\label{residual}
R_{\Gamma} (v) = \epsilon (\frac{\partial v}{\partial
n}-g )+v-u_{0 } ,
\end{equation}
the boundary condition is
\begin{equation}
R_{\Gamma}(u)=0.\end{equation}
Hence it holds  
\begin{equation}
\label{e:consist2}
 \sum_{E \in \mathcal{G}_h} \frac{1}{\epsilon+\gamma h_E}  \big< R_{\Gamma}(u),w
\big>_E  
  =
0
\end{equation}
and 
\begin{equation}
\label{e:consist3}
 -\sum_{E \in \mathcal{G}_h} \frac{\gamma h_E}{\epsilon+\gamma h_E}
  \big< R_{\Gamma}(u),\frac{\partial w}{\partial n}\big>_E=0.
\end{equation}

 Adding \eqref{I}, \eqref{e:consist2} and \eqref{e:consist3} shows that the exact solution satisfies
 \begin{eqnarray}\label{A}
   \nonumber
  \big(\nabla u, \nabla w\big)_{\Omega}-\big<\frac{\partial u}{\partial n}, w\big>_{\Gamma}- \big(f, w\big)_{\Omega}&+&
 \sum_{E \in \mathcal{G}_h} \frac{1}{\epsilon+\gamma h_E}  \big< R_{\Gamma}(u),w
\big>_E  
\\&
-&\sum_{E \in \mathcal{G}_h} \frac{\gamma h_E}{\epsilon+\gamma h_E}
  \big< R_{\Gamma}(u),\frac{\partial w}{\partial n}\big>_E=0.
 \end{eqnarray}
 Substituting the expression \eqref{residual} for the boundary condition and rearranging the terms, we see that the exact solution satis\-fies 
\begin{equation}\label{AA}
\mathcal{B}_h(u,w)-\mathcal{F}_h(w) =0 \quad \forall w \in V_h
\end{equation}
where
\begin{multline}\
\mathcal{B}_h(v,w) = \ \big(\nabla v, \nabla w\big)_{\Omega}
+ \sum_{E \in \mathcal{G}_h}
\Bigg\{
-\frac{\gamma h_E}{\epsilon+\gamma h_E}
\Big[
\ \big<\frac{\partial v}{\partial n},w\big>_E
+\big<v,\frac{\partial w}{\partial n}\big>_E
\Big] \\
\ + \frac{1}{\epsilon + \gamma h_E} \, \big<v,w\big>_E
- \frac{\epsilon \gamma h_E}{\epsilon + \gamma h_E}\,
\big<\frac{\partial v}{\partial n},
\frac{\partial w}{\partial n}\big>_E
\Bigg\}
\end{multline}
and
\begin{multline}\label{B}
\mathcal{F}_h(w) = \ \big(f,w\big)_\Omega + \sum_{E \in
\mathcal{G}_h}
\Bigg\{
\frac{1}{\epsilon+\gamma h_E} \big<u_0,w\big>_E
-\frac{\gamma h_E}{\epsilon + \gamma h_E}
\big<u_0,\frac{\partial w}{\partial n} \big>_E
\\
 \ +\frac{\epsilon}{\epsilon+\gamma h_E} \big<g,w\big>_E
-\frac{\epsilon \gamma h_E}{\epsilon+\gamma h_E}
\big<g,\frac{\partial w}{\partial n}\big>_E
\Bigg\} .
\end{multline}

The above derivation shows the consistency of the \medskip
\noindent

{\bf Nitsche Method \cite{just}.} {\it Find $u_h \in V_h$ such that}
\begin{equation}
\label{e:method}
\mathcal{B}_h(u_h,w) = \mathcal{F}_h(w) \quad \forall w \in V_h.
\end{equation}

\section{The new a priori error estimate}
  The estimate will be given in the mesh and problem dependent norm  
\begin{equation}\label{normdef}
\| v \|_{h}^2  :=  \| \nabla v \|_{0,\Omega}^2
+\sum_{E \in \mathcal{G}_h}
\frac{1}{\epsilon+h_E} \| v \|_{0,E}^2 .
\end{equation}
 
 We recall the following discrete trace inequality which is easily proved by scaling arguments.
\begin{lemma}
There is a positive constant $C_I$ such that
\begin{equation}
\label{e:CI}
\sum_{E \in \mathcal{G}_h} h_E
\left\| \frac{\partial v}{\partial n} \right\|_{0,E}^2
\le C_I \| \nabla v \|_{0,\Omega}^2 \quad \forall v \in V_h .
\end{equation}
\end{lemma}
For the formulation we have the following stability result, cf. \cite{just}. Here and
in what follows $C$ denotes a generic positive constant independent
of both the mesh parameter $h$ and the parameter $\epsilon$.
\begin{lemma}
\label{l:stability} Suppose that $  0< \gamma <1/C_I$. Then there
exists a positive constant $C$ such that
\begin{equation}
\mathcal{B}_h(v,v) \ge C \| v \|_h^2 \quad \forall v \in V_h.
\end{equation}
\end{lemma} 

By $f_{h}\in V_{h}$ and $g_{h}, \, u_{0,h}\in V_{h} \vert _\Gamma$  we denote the interpolants to the data. 
For $E\in \mathcal{G}_{h}$ we denote by $K(E)\in \mathcal{T}_{h}$ the element with $E$ as edge/face.  
In \cite{just} we proved the following bound. 

\begin{lemma} \label{lemma23}For an arbitrary $v\in V_{h}$ and $E \in \mathcal{G}_{h}$ it holds
\begin{multline}
\label{e:rgok} \frac{h_E^{1/2}}{\epsilon +h_E} \| R_{\Gamma}(v)
\|_{0,E}
 \le C \Big(
\| \nabla(u-v )\|_{0,K(E)} +h_K \| f-f_h \|_{0,K(E)} \\
+\frac{1}{(\epsilon +h_E)^{1/2}} \| u-v \|_{0,E}
+\frac{h_E^{1/2}}{\epsilon +h_E} \| \epsilon (g-g_h) +u_0-u_{0,h}
\|_{0,E} \Big).
\end{multline}
\end{lemma}
We introduce the oscillation terms
\begin{eqnarray}\label{oscf}
\osc(f)&=& \Big(\sum_{ K\in \mathcal{T}_h} h_{K}^{2} \| f-f_h \|_{0,K}^{2}\Big)^{1/2},
\\
\label{oscg}
\osc(\epsilon,u_0,g)&=&\Big( \sum_{E \in \mathcal{G}_h} 
\frac{h_E }{(\epsilon +h_E)^{2}} \| \epsilon (g-g_h) +u_0-u_{0,h}
\|_{0,E} ^{2}\Big)^{1/2}.
\end{eqnarray}
Lemma \ref{lemma23} then gives
\begin{lemma}
\label{resest} 
For $v\in V_{h}$ it holds
\begin{equation}
\Big(\sum_{E\in \mathcal{G}_{h}}
\frac{h_E}{(\epsilon +h_E)^{2}} \| R_{\Gamma}(v)
 \|_{0,E}^{2}\Big)^{1/2}
 \leq C \big\{  \| u-v \|_h+\osc(f)+\osc(\epsilon,u_{0},g)\big\}.
\end{equation}
\end{lemma}
 We can now prove our new error estimate.
\begin{theorem} Suppose that $0<\gamma <1/C_{I}.$ Then there exist a positive constant C such that
\begin{equation}
\label{e:ap2}
\| u-u_h \|_h
 \leq C \big\{\inf_{v\in V_{h}}   \| u-v \|_h+\osc(f)+\osc(\epsilon,u_{0},g)\big\}.
\end{equation}
\end{theorem}

\begin{proof} We will divide the proof in 6 steps.

\smallskip
\noindent 1. {\it Treating the  consistency   by Gudi's method.}

Let  $v\in V_{h} $ be arbitrary. From the stability we have
\begin{equation}
 C \|v -u_h\|_h^2 \leq  \mathcal{B}_h(v -u_h,v -u_h). 
\end{equation}
Next, we denote $w=v -u_h$ and use \eqref{e:method}
\begin{eqnarray}
 \mathcal{B}_h(v -u_h,v -u_h) &=& \mathcal{B}_h(v -u_h,w)
 = \mathcal{B}_h(v ,w)-\mathcal{B}_h(u_{h},w)
 \\
 &=& \mathcal{B}_h(v ,w)-\mathcal{F}_h(w) . \nonumber
\end{eqnarray}
Reversing the arguments leading from \eqref{A} to \eqref{AA} we see that
\begin{eqnarray}\nonumber \label{II}
   \mathcal{B}_h(v ,w)-\mathcal{F}_h(w)& =&  \big(\nabla v, \nabla w\big)_{\Omega}-\big<\frac{\partial v}{\partial n}, w\big>_{\Gamma}- \big(f, w\big)_{\Omega}\\&+&
 \sum_{E \in \mathcal{G}_h} \frac{1}{\epsilon+\gamma h_E}  \big< R_{\Gamma}(v),w
\big>_E  
\\&
-&\sum_{E \in \mathcal{G}_h} \frac{\gamma h_E}{\epsilon+\gamma h_E}
  \big< R_{\Gamma}(v),\frac{\partial w}{\partial n}\big>_E .
  \nonumber
 \end{eqnarray}
 Substituting the boundary condition \eqref{e:modelbcond} into \eqref{I} we get
 \begin{equation}\label{III}
 \big(\nabla u, \nabla w\big)_{\Omega}-\big<\frac{1}{\epsilon}(u_0-u) + g, w\big>_{\Gamma}- \big(f, w\big)_{\Omega}=0.
\end{equation}
Subtracting this from the right hand side of \eqref{II} yields
\begin{eqnarray}\nonumber \label{split}
   \mathcal{B}_h(v ,w)-\mathcal{F}_h(w)& =&  \big(\nabla (v-u), \nabla w\big)_{\Omega}
   -\big<\frac{\partial v}{\partial n}-\frac{1}{\epsilon}(u_0-u) - g, w\big>_{\Gamma} \\&+&
 \sum_{E \in \mathcal{G}_h} \frac{1}{\epsilon+\gamma h_E}  \big< R_{\Gamma}(v),w
\big>_E  
\nonumber
\\&
-&\sum_{E \in \mathcal{G}_h} \frac{\gamma h_E}{\epsilon+\gamma h_E}
  \big< R_{\Gamma}(v),\frac{\partial w}{\partial n}\big>_E 
  \\
  & =& R_{1}+R_{2}+R_{3}+R_{4}.
  \nonumber
 \end{eqnarray}
Next we estimate the terms in the right hand side above. 

\smallskip

\noindent2. {\it Estimates for the terms $R_{1}$ and $R_{4}$.}

The first and the last term are readily estimated. By Schwarz inequality and the definition  \eqref{normdef} of the norm, we have
\begin{equation}\label{r1}
R_{1}=\big(\nabla (v-u), \nabla w\big)_{\Omega}\leq \| u-v \|_h \| w \|_h.
\end{equation}
Schwarz inequality, the discrete trace inequality \eqref{e:CI}, and Lemma \ref{resest} give
\begin{align}
\label{r4}
\nonumber 
R_{4} & \leq\Big\vert \sum_{E \in \mathcal{G}_h} \frac{\gamma h_E}{\epsilon+\gamma h_E}
  \big< R_{\Gamma}(v),\frac{\partial w}{\partial n}\big>_E \Big\vert 
\\
& \leq
  \Big(\sum_{E \in \mathcal{G}_h}
   \frac{\gamma^{2} h_E}{(\epsilon+\gamma h_E)^{2}}
 \Vert R_{\Gamma}(v)\Vert_{0,E}^{2}\Big)^{1/2}
\Big( \sum_{E \in \mathcal{G}_h}   h_E
\Vert  \frac{\partial w}{\partial n}\Vert_{0,E}^{2}\Big)^{1/2}
\\
&\leq 
C \big(  \| u-v \|_h+\osc(f)+\osc(\epsilon,u_{0},g)\big)\Vert w\Vert_{h} .
\nonumber 
\end{align}

\noindent 3. {\it Splitting the boundary.}

To treat the two remaining terms $ R_{2}$ and $ R_{3}$, we have to separate the cases when the edge size $h_{E}$ is smaller or greater than $\epsilon$.  To this end we denote the collection of edges of size greater than $\epsilon $  by 
\begin{equation}
 \mathcal G_h^\epsilon = \{ E \in \mathcal G _h \mid \epsilon < h_E \},
 \end{equation}
 and the corresponding part of the boundary by
 \begin{equation}
 \Gamma_\epsilon  = \bigcup_{E \in \mathcal G_h^\epsilon} E.
\end{equation}
We then write
\begin{align}\nonumber
R_{2} + R_{3}& =
    -\big<\frac{\partial v}{\partial n}-\frac{1}{\epsilon}(u_0-u) - g, w\big>_{\Gamma} 
 +\sum_{E \in \mathcal{G}_h} \frac{1}{\epsilon+\gamma h_E}  \big< R_{\Gamma}(v),w
\big>_E
\\
&   = \sum_{E \in \mathcal{G}_h} \Big\{  -\big<\frac{\partial v}{\partial n}-\frac{1}{\epsilon}(u_0-u) - g, w\big>_{E}  + \frac{1}{\epsilon+\gamma h_E}  \big< R_{\Gamma}(v),w
\big>_E \Big\}
\\
\nonumber
&=   \sum_{E \subset  \Gamma_{\epsilon}} \Big\{  -\big<\frac{\partial v}{\partial n}-\frac{1}{\epsilon}(u_0-u) - g, w\big>_{E}  + \frac{1}{\epsilon+\gamma h_E}  \big< R_{\Gamma}(v),w
\big>_E \Big\}
\\&\quad +    \sum_{E \subset \Gamma \setminus \Gamma_{\epsilon}} \Big\{  -\big<\frac{\partial v}{\partial n}-\frac{1}{\epsilon}(u_0-u) - g, w\big>_{E}  + \frac{1}{\epsilon+\gamma h_E}  \big< R_{\Gamma}(v),w
\big>_E \Big\}.
\nonumber 
 \end{align}

\medskip
\noindent 4. {\it Estimation of the contribution to $R_{2}+R_{3}$ from the part $\Gamma_\epsilon$.}

 On $E\subset \Gamma_{\epsilon} $ it holds $\epsilon < h_{E}$ and  we   estimate as follows, using Lemma \ref{resest},
 \begin{align}\label{c}
 \nonumber 
    \sum_{E \subset  \Gamma_{\epsilon}} & \frac{1}{\epsilon+\gamma h_E}  \big< R_{\Gamma}(v),w
\big>_E\\
\nonumber
& \leq 
 \sum_{E \subset  \Gamma_{\epsilon}} \frac{(\epsilon +h_{E})^{1/2}}{\epsilon+\gamma h_E}  \Vert R_{\Gamma}(v)\Vert_{0,E}\cdot(\epsilon +h_{E})^{-1/2}\Vert w \Vert_{0,E}
\\&\leq  \sum_{E \subset  \Gamma_{\epsilon}}    
\frac{  \sqrt{2}  h_{E}^{1/2}}{\epsilon+\gamma h_E}  \Vert R_{\Gamma}(v)\Vert_{0,E}\cdot(\epsilon +h_{E})^{-1/2}\Vert w \Vert_{0,E}
\\ 
\nonumber
 &\leq 
\Big(  \sum_{E \subset  \Gamma_{\epsilon}}  
\frac{2h_{E}}{(\epsilon+\gamma h_E)^{2}}  \Vert R_{\Gamma}(v)\Vert_{0,E}^{2}\Big)^{1/2}\Big(
 \sum_{E \subset  \Gamma_{\epsilon}} (\epsilon +h_{E})^{-1}\Vert w \Vert_{0,E}^{2}
\Big)^{1/2}
\\
\nonumber
&\leq
C\big  ( \| u-v \|_h+\osc(f)+\osc(\epsilon,u_{0},g)\big)\Vert w \Vert_{h}.
 \nonumber 
 \end{align}
Next, we have to estimate
\begin{equation}
 -\big<\frac{\partial v}{\partial n}-\frac{1}{\epsilon}(u_0-u) - g, w\big>_{\Gamma_{\epsilon}} .
\end{equation}
We substitute 
\begin{equation}
 \frac{1}{\epsilon}(u_0-u) + g= \frac{\partial u}{\partial n}  ,
\end{equation}
which gives
\begin{equation}
 -\big<\frac{\partial v}{\partial n}-\frac{1}{\epsilon}(u_0-u) -g, w\big>_{\Gamma_{\epsilon}} = 
 \big< \frac{\partial u}{\partial n}-\frac{\partial v}{\partial n}, w \big>_{\Gamma_\epsilon} .
\end{equation}
Now we define the strip
\begin{equation}
 \Omega_\epsilon  = \bigcup_{\substack{K \in \mathcal T_h \\  K \cap \Gamma_\epsilon \neq \emptyset}} K.
 \end{equation}
Following \cite{AK,ccea,MR1972728} we construct a linear finite element extension $\exth w \in V_{h} $ of $w\vert_{\Gamma_{\epsilon}}$ such that
\begin{equation}
 \exth w  \vert_{\Gamma_{\epsilon}}=w\vert_{\Gamma_{\epsilon}} 
\end{equation}
and 
\begin{equation}
\exth w=0 \ \mbox{ in } \  \Omega\setminus \Omega_{\epsilon}.
\end{equation}
In \cite{AK,ccea,MR1972728} the following estimate is derived
\begin{equation}\label{trace}
\Vert \nabla  \exth w \Vert_{0,\Omega_{{\epsilon}}}\leq C\big(\sum_{E\subset \Gamma_{\epsilon}}h_{E}^{-1}\Vert w \Vert_{0,E}^{2}\big)^{1/2}.
\end{equation}
We denote
 \begin{equation}
 \Gamma_\epsilon^+  = \Omega_\epsilon \cap \Gamma.
\end{equation}
and 
split the boundary of $ \Omega_{\epsilon} $ in three parts (cf. Figure \ref{figure})
\begin{equation}
 \partial \Omega_{\epsilon} = \Gamma_{\epsilon}\cup \{\Gamma_{\epsilon}^{+}\setminus \Gamma_{\epsilon}\}\cup\{
  \partial \Omega_{\epsilon}\setminus �\Gamma_{\epsilon}^{+}\} .
\end{equation}

\begin{figure}\label{figure}
\includegraphics[scale=0.6]{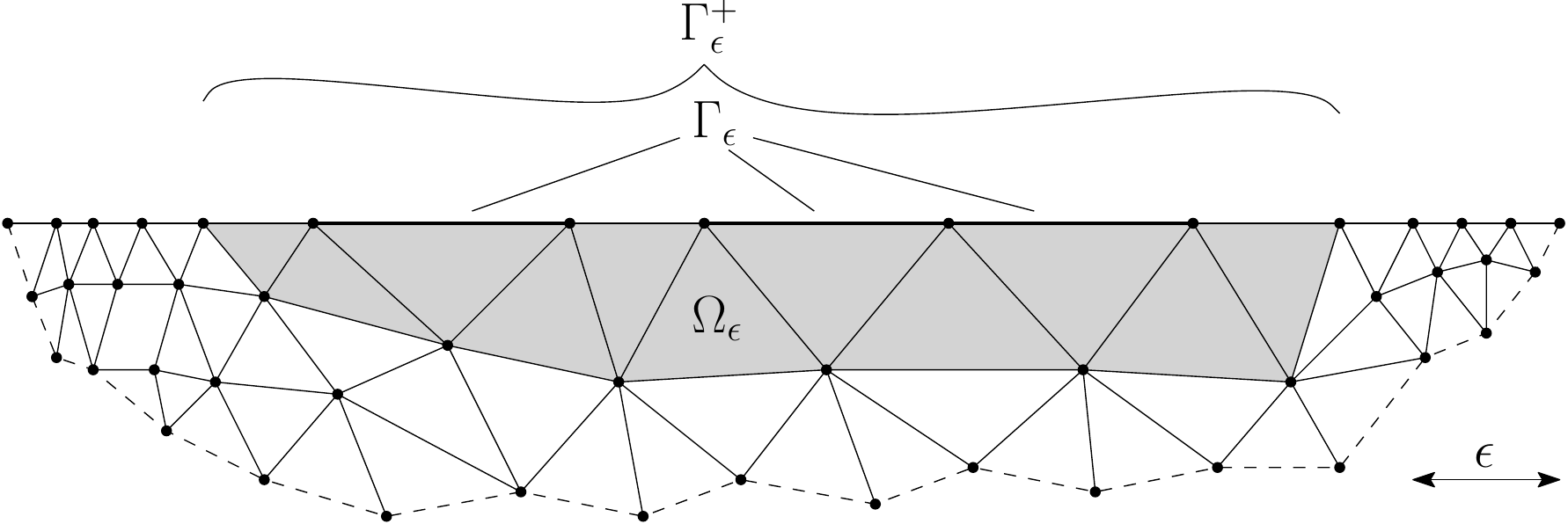}
\caption{The boundary parts   $ \Gamma_{\epsilon} \mbox{ and }  \Gamma_{\epsilon}^{+}$, and the strip $\Omega_{\epsilon} $.}
\end{figure}

Note that  $ \exth w \not = w $ on $ \Gamma_{\epsilon}^{+}\setminus \Gamma_{\epsilon}$.
Since $\exth w\vert_{ \partial \Omega_{\epsilon}\setminus \Gamma_{\epsilon}^{+}}=0$, scaling and the estimate \eqref{trace} show that
\begin{align}\label{z}
  \big(  &\sum_{K \subset \Omega_\epsilon}h_K^{-2} \Vert \exth w \Vert_{0,K}^2 \big)^{1/2} 
 + \big(\sum_{E \subset  \Omega_\epsilon \setminus  \Gamma_\epsilon^+} h_E^{-1} \Vert \exth w \Vert_{0,E}^2 \big)^{1/2} 
  \\& \leq 
C \Vert \nabla \exth w \Vert_{0,\Omega_\epsilon}\leq C\big(\sum_{E\subset \Gamma_{\epsilon}}h_{E}^{-1}\Vert w \Vert_{0,E}^{2}\big)^{1/2},
\nonumber\end{align}
and also
\begin{equation}\label{zz}
  \big(\sum_{E \subset   \Gamma_\epsilon^+\setminus \Gamma_\epsilon} h_E^{-1} \Vert \exth w \Vert_{0,E}^2 \big)^{1/2} 
  \leq 
C \Vert \nabla \exth w \Vert_{0,\Omega_\epsilon} \leq C\big(\sum_{E\subset \Gamma_{\epsilon}}h_{E}^{-1}\Vert w \Vert_{0,E}^{2}\big)^{1/2}.
 \end{equation}
Further, since $\epsilon < h_{E}$, it holds
\begin{equation}\label{zzz}
 \big(\sum_{E\subset \Gamma_{\epsilon}}h_{E}^{-1}\Vert w \Vert_{0,E}^{2}\big)^{1/2}\leq 
\sqrt{2} \big( \sum_{E \subset \Gamma_\epsilon} \frac 1 {h_E + \epsilon} \Vert w \Vert_{0,E}^2 \big)^{1/2}
 \leq \sqrt{2}   \Vert w \Vert_h. 
\end{equation} 
Next, integrating by parts and using \eqref{z}--\eqref{zzz} we estimate as follows
\begin{align}\nonumber
 &\hskip -15 pt \big< \frac{\partial u}{\partial n}-\frac{\partial v}{\partial n}, \exth w \big>_{\Gamma_\epsilon^+} \\
 &= \sum_{K \subset \Omega_\epsilon} \left[ - \big( f + \Delta v, \exth w \big)_K + \big( \nabla(u-v), \nabla \exth w \big)_K \right] 
 \\
 \nonumber
 & \quad + \sum_{E \subset \Omega_\epsilon \setminus  \Gamma_\epsilon^+}  \big< \blj \frac{\partial v}{\partial n} \brj, \exth w \big>_E \\
 \nonumber
 &\leq C \big( \sum_{K \subset \Omega_\epsilon} h_K^2 \Vert f + \Delta v \Vert_{0,K}^2 \big)^{1/2} 
           \big( \sum_{K \subset \Omega_\epsilon} h_K^{-2} \Vert \exth w \Vert_{0,K}^2 \big)^{1/2} 
 \\
 \nonumber
 &\quad +  \Vert \nabla (u-v) \Vert_{0,\Omega_\epsilon} \Vert \nabla \exth w \Vert_{0,\Omega_\epsilon} 
 \\
 \nonumber
 &\quad + \big( \sum_{E \subset \Omega_\epsilon \setminus  \Gamma_\epsilon^+} h_E \Vert \blj \frac{\partial v}{\partial n} \brj \Vert_{0,E}^2 \big)^{1/2} 
    \big( \sum_{E \subset \Omega_\epsilon \setminus  \Gamma_\epsilon^+}  h_E^{-1} \Vert \exth w \Vert_{0,E}^2 \big)^{1/2} 
 \\
 \nonumber
 &\leq C\Big\{ \big(\sum_{K \subset \Omega_\epsilon} h_K^2 \Vert f + \Delta v \Vert_{0,K}^2 \big)^{1/2} 
  + \Vert \nabla (u-v) \Vert_{0,\Omega_\epsilon} \\
 \nonumber
 &\quad + \big( \sum_{E \subset \Omega_\epsilon \setminus  \Gamma_\epsilon^+} h_E \Vert \blj \frac{\partial v}{\partial n} \brj \Vert_{0,E}^2 \big)^{1/2} \Big\} \Vert w \Vert_{h}.
\end{align}
From a posteriori error analysis \cite{braess,verfurth} we know that
\begin{align}\label{zzzz}
\big(  &\sum_{E \subset \Omega_\epsilon \setminus  \Gamma_\epsilon^+}  h_E \Vert \blj \frac{\partial v}{\partial n} \brj \Vert_{0,E}^2 \big)^{1/2}
\\
& \leq \notag
C \big( \sum_{K \subset \Omega_\epsilon} h_K^2 \Vert f + \Delta v \Vert_{0,K}^2 \big)^{1/2} +  \Vert \nabla (u-v) \Vert_{0,\Omega_\epsilon} 
\end{align}
and
\begin{equation}
 \big( \sum_{K \subset \Omega_\epsilon} h_K^2 \Vert f + \Delta v \Vert_{0,K}^2 \big)^{1/2} \leq C\big( \Vert \nabla (u-v) \Vert_{0,\Omega } +\osc(f)\big).
\end{equation}
Hence we have
\begin{align}
 \big< \frac{\partial u}{\partial n}-\frac{\partial v}{\partial n},E_{h} w \big>_{\Gamma_\epsilon^+} \leq C \Big(\Vert u-v \Vert_h + \osc(f) \Big) \Vert w \Vert_h.
\end{align}
Since $\exth w=w$ on $\Gamma_{\epsilon}$,  we get
\begin{align}
 \big< \frac{\partial u}{\partial n}-\frac{\partial v}{\partial n}, w \big>_{\Gamma_\epsilon} 
 \leq  C \Big(\Vert u-v \Vert_h + \osc(f) \Big) \Vert w \Vert_h
 - \big< \frac{\partial u}{\partial n}-\frac{\partial v}{\partial n}, \exth w \big>_{\Gamma_\epsilon^+\setminus \Gamma_\epsilon}.
\end{align}
For $E \subset \Gamma_\epsilon^+\setminus \Gamma_\epsilon$ it holds that $h_E \leq \epsilon \leq C h_E$ 
with a constant that only depends on the regularity constant $\kappa$. Thus we can estimate
\begin{align}\notag 
 &\hskip -15 pt - \big< \frac{\partial u}{\partial n}-\frac{\partial v}{\partial n}, \exth w \big>_{\Gamma_\epsilon^+\setminus \Gamma_\epsilon}
  \\
  \notag
 &= \sum_{E \subset \Gamma_\epsilon^+\setminus \Gamma_\epsilon} - \frac 1 \epsilon \big< R_\Gamma(v), \exth w \big>_E - \frac 1 \epsilon \big< u-v, \exth w \big>_E
  \\
 &\leq C \Big(\sum_{E \subset \Gamma_\epsilon^+\setminus \Gamma_\epsilon} 
          \frac 1 {\epsilon + h_E} \Vert R_\Gamma(v) \Vert_{0,E} \Vert \exth w \Vert_{0,E} + \frac 1 {\epsilon + h_E} \Vert u - v \Vert_{0,E} \Vert \exth w \Vert_{0,E}\Big) 
          \\
          \notag
 &\leq C \big(  \sum_{E \subset \Gamma_\epsilon^+\setminus \Gamma_\epsilon} \frac{h_E}{(h_E + \epsilon)^2} \Vert R_\Gamma(v) \Vert_{0,E}^2 
                   + \frac 1 {h_E + \epsilon} \Vert u-v \Vert_{0,E}^2 \big)^{1/2} 
                   \\
                   \notag  &\qquad \times 
 \big(  \sum_{E \subset \Gamma_\epsilon^+\setminus \Gamma_\epsilon} \frac 1 {h_E + \epsilon} \Vert \exth w \Vert_{0,E}^2 \big)^{1/2}.
\end{align}
Thus we have
\begin{align}
 \big< \frac{\partial u}{\partial n}-\frac{\partial v}{\partial n}, w \big>_{\Gamma_\epsilon} \leq C \big(\| u-v \|_h+\osc(f)+\osc(\epsilon,u_{0},g)\big) \Vert w \Vert_h,
\end{align}
which  together with \eqref{c}
gives
\begin{align}\label{cc}\notag
 \sum_{E \subset  \Gamma_{\epsilon}} \Big\{  -&\big<\frac{\partial v}{\partial n}-\frac{1}{\epsilon}(u_0-u) - g, w\big>_{E}  
         + \frac{1}{\epsilon+\gamma h_E}  \big< R_{\Gamma}(v),w
 \big>_E \Big\}
 \\
 &\leq C \big(\| u-v \|_h+\osc(f)+\osc(\epsilon,u_{0},g)\big) \Vert w \Vert_h.
\end{align}

\noindent 5. {\it Estimation of the contribution to $R_{2}+R_{3}$ from the part  $\Gamma\setminus \Gamma_\epsilon$.}
\smallskip
 
It now holds $\epsilon \geq h_{E}$. First write 
\begin{align}
 \frac{\partial v}{\partial n}-\frac{1}{\epsilon}(u_0-u) - g &= \frac{\partial v}{\partial n}-\frac{1}{\epsilon}(u_0-v) - g +\frac{1}{\epsilon}(u-v)
 \\&= \frac{1}{\epsilon}R_{\Gamma}(v)+\frac{1}{\epsilon}(u-v).
 \nonumber
\end{align}
Hence, on $E$ it holds
\begin{equation}\begin{aligned}\label{o}
 &\hskip -15 pt -\big<\frac{\partial v}{\partial n}-\frac{1}{\epsilon}(u_0-u) - g, w\big>_{E}  + \frac{1}{\epsilon+\gamma h_E}  \big< R_{\Gamma}(v),w\big>_E 
 \\
 &= \big( \frac{1}{\epsilon+\gamma h_E} - \frac{1}{\epsilon}  \big) \big< R_{\Gamma}(v),w \big>_E - \frac{1}{\epsilon} \big<u-v ,w\big>_{E} 
 \\
 &=
 -\frac{\gamma h_{E}}{(\epsilon+\gamma h_E)\epsilon}   \big< R_{\Gamma}(v),w \big>_E -\frac{1}{\epsilon} \big<u-v, w\big>_{E}
 \\
 &\leq  \frac{\gamma h_{E}}{(\epsilon+\gamma h_E)\epsilon}   \Vert R_{\Gamma}(v) \Vert_{0,E}\Vert w \Vert_{0,E} 
      +\frac{1}{\epsilon}\Vert u-v  \Vert_{0,E}\Vert w\Vert_{0,E}.
 \end{aligned}
\end{equation}
Since $\epsilon+h_{E}\leq 2 \epsilon$, it holds
\begin{equation}
 \frac{1}{\epsilon}\Vert u-v  \Vert_{0,E}\Vert w \Vert_{0,E} \leq \frac{2}{\epsilon+h_{E}}\Vert u-v  \Vert_{0,E}\Vert w\Vert_{0,E}.
\end{equation}
Since $h_{E}/\epsilon\leq1$ we estimate as follows
\begin{equation}\label{oo}
 \begin{aligned}  
 &\frac{\gamma h_{E}}{(\epsilon+\gamma h_E)\epsilon}   \Vert R_{\Gamma}(v) \Vert_{0,E}\Vert w\Vert_{0,E} 
 \\
 &= \frac{\gamma h_{E} (\epsilon+h_{E})^{1/2}}{(\epsilon+\gamma h_E)\epsilon}   \Vert R_{\Gamma}(v)\Vert_{0,E}\cdot  (\epsilon+h_{E})^{-1/2} \Vert w \Vert_{0,E}
 \\
 &= \gamma \frac{h_{E}^{1/2}}{\epsilon^{1/2}} \frac{h_{E}^{1/2}}{(\epsilon+\gamma h_E) }\frac{(\epsilon+h_{E})^{1/2}}{ \epsilon^{1/2}} \Vert R_{\Gamma}(v)\Vert_{0,E}
    \cdot  (\epsilon+h_{E})^{-1/2} \Vert w \Vert_{0,E}
 \\
 &= \gamma \frac{h_{E}^{1/2}}{\epsilon^{1/2}} \frac{h_{E}^{1/2}}{(\epsilon+\gamma h_E) } (1+\frac{h_{E}}{\epsilon})^{1/2} \Vert R_{\Gamma}(v)\Vert_{0,E}
    \cdot  (\epsilon+h_{E})^{-1/2} \Vert w \Vert_{0,E}
 \\
 &\leq \sqrt{2} \gamma  \frac{h_{E}^{1/2}}{(\epsilon+\gamma h_E) }   \Vert R_{\Gamma}(v)\Vert_{0,E}\cdot  (\epsilon+h_{E})^{-1/2} \Vert w \Vert_{0,E}.
 \end{aligned}
\end{equation}
Combining \eqref{o}--\eqref{oo}  yields 
\begin{align}\label{ooo}
 &\hskip -15 pt \sum_{E\subset \Gamma\setminus \Gamma_{\epsilon}}\Big(-\big<\frac{\partial v}{\partial n}-\frac{1}{\epsilon}(u_0-u) - g, w\big>_{E}  + \frac{1}{\epsilon+\gamma h_E}  \big< R_{\Gamma}(v),w\big>_E \Big)
 \\
 &\leq C \big(\| u-v \|_h+\osc(f)+\osc(\epsilon,u_{0},g)\big) \Vert w \Vert_h.
 \notag
\end{align}
\smallskip
\noindent 6. {\it Collecting the estimates.}
\smallskip

Adding \eqref{cc}  and \eqref{ooo} gives
\begin{equation}
R_{2}+R_{3}\leq C   \big(\| u-v \|_h+\osc(f)+\osc(\epsilon,u_{0},g)\big) \Vert w \Vert_h.
\end{equation}
The assertion then follows from this and \eqref{split}, \eqref{r1}, and \eqref{r4}.

\end{proof} 

\bibliography{LJS}
\bibliographystyle{siam}
\end{document}